\documentclass[11pt]{amsart}

\usepackage[colorlinks=true, pdfstartview=FitV, linkcolor=blue, citecolor=blue, urlcolor=blue, breaklinks=true]{hyperref}
\usepackage{amsmath,amsfonts,amssymb,amsthm,mathrsfs,amscd,comment,paralist,stmaryrd,etoolbox,mathtools,enumitem}
\usepackage[usenames,dvipsnames]{color}
\usepackage{booktabs}
\usepackage[all]{xy}
\usepackage{tikz}
\usetikzlibrary{decorations.pathmorphing}

\theoremstyle{plain}
\newtheorem{thm}{Theorem}[section]
\newtheorem{prop}[thm]{Proposition}
\newtheorem{lemma}[thm]{Lemma}
\newtheorem{cor}[thm]{Corollary}

\theoremstyle{definition}
\newtheorem{defn}[thm]{Definition}
\newtheorem{rmk}[thm]{Remark}

\numberwithin{equation}{section}

\newcommand{\g}{\mathfrak{g}}
\newcommand{\h}{\mathfrak{h}}
\newcommand{\bb}{\mathfrak{b}}

\newcommand{\n}{\mathfrak{n}}

\newcommand{\gl}{\mathfrak{gl}}

\newcommand{\fsl}{\mathfrak{sl}}

\newcommand{\I}{\mathcal{I}}
\newcommand\cL{\mathcal{L}}

\newcommand{\Z}{\mathbb{Z}}
\newcommand{\C}{\mathbb{C}}
\newcommand{\N}{\mathbb{N}}

\newcommand{\sm}{\mathsf{m}}

\DeclareMathOperator{\Hom}{Hom}

\DeclareMathOperator{\vspan}{span}

\DeclareMathOperator{\MaxSpec}{MaxSpec}

\DeclareMathOperator{\Supp}{Supp}

\DeclareMathOperator{\rank}{rank}

\DeclareMathOperator{\het}{ht}

\newtoggle{comments}
\newtoggle{details}
\newtoggle{prelimnote}
\newtoggle{detailsnote}

\toggletrue{detailsnote}   

\iftoggle{comments}{%
  \newcommand{\comments}[1]{
    \begin{center}
      \parbox{6.5 in}{
        \color{red}
          {\footnotesize \textbf{Comments:} #1}
        \color{black}}
    \end{center}}
}{%
  \newcommand{\comments}[1]{}
}

\iftoggle{details}{%
  \newcommand{\details}[1]{
      \ \\
      \color{OliveGreen}
        \begin{footnotesize}
          \textbf{Details:} #1
        \end{footnotesize}
      \color{black}
      \\
  }
}{%
  \newcommand{\details}[1]{}
}

\iftoggle{prelimnote}{%
  
}{%
  
}

\begin{document}
%

\title{Weyl modules for Lie superalgebras}

\author{Lucas Calixto}
\address{L.~Calixto: Department of Mathematics, Federal University of Minas Gerais, Belo Horizonte, Minas Gerais, Brazil}
\urladdr{\url{http://www.mat.ufmg.br/lhcalixto/}}
\email{lhcalixto@ufmg.br}

\author{Joel Lemay}
\address{J.~Lemay: Department of Mathematics and Statistics, University of Ottawa, Canada}
\email{jlema072@uottawa.ca}

\author{Alistair Savage}
\address{A.~Savage: Department of Mathematics and Statistics, University of Ottawa, Canada}
\urladdr{\url{https://alistairsavage.ca}}
\email{alistair.savage@uottawa.ca}

\thanks{L.~Calixto was supported by  FAPESP grant 2013/08430-4.  J.~Lemay was supported by a Natural Sciences and Engineering Research Council of Canada (NSERC) Postgraduate Scholarship.  A.~Savage was supported by an NSERC Discovery Grant.}

\begin{abstract}
  We define global and local Weyl modules for Lie superalgebras of the form $\g \otimes A$, where $A$ is an associative commutative unital $\mathbb{C}$-algebra and $\g$ is a basic Lie superalgebra or $\fsl(n,n)$, $n \ge 2$.  Under some mild assumptions, we prove universality, finite-dimensionality, and tensor product decomposition properties for these modules.  These properties are analogues of those of Weyl modules in the non-super setting.  We also point out some features that are new in the super case.
\end{abstract}

\subjclass[2010]{17B65, 17B10.}
\keywords{Lie superalgebra, basic Lie superalgebra, Weyl module, Kac module, tensor product.}

\maketitle
\thispagestyle{empty}

\setcounter{tocdepth}{1}

\tableofcontents

%
\section{Introduction}
%

Map Lie algebras, also known as generalized current Lie algebras, are Lie algebras of regular maps from a scheme $X$ to a (generally finite-dimensional) target Lie algebra $\g$.  They form a large class of Lie algebras that include the important loop and current algebras as special cases, and their representation theory is an active area of research.  We refer the reader to \cite{NS13} for a survey of the field.  A vital ingredient in the theory is played by global and local Weyl modules, which are universal objects with respect to certain highest weight properties.  The local Weyl modules are finite-dimensional but not, in general, irreducible.  They were first defined, in the loop case, in \cite{CP01} and extended to the map case in \cite{FL04}.

Replacing the target Lie algebra $\g$ by a Lie \emph{super}algebra, we obtain the class of map superalgebras.  The study of these algebras is still in its infancy.  In the loop case, where $X$ is a torus, and when $\g$ is a basic Lie superalgebra, the finite-dimensional modules were classified in \cite{RZ04,Rao13}.  In the more general setting where the coordinate ring of $X$ is finitely generated, $\g$ is a basic Lie superalgebra, and where we also consider maps equivariant with respect to a finite abelian group acting freely on the rational points of $X$, the irreducible finite-dimensional modules were classified in \cite{Sav14}.  However, Weyl modules have not been defined in the super setting, except for a quantum analogue in the loop case for $\g = \fsl(m,n)$ considered in \cite{Zha14}.

In the current paper, we initiate the study of Weyl modules for Lie superalgebras.  In particular, we consider Lie superalgebras of the form $\g \otimes_\C A$, where $A$ is an associative commutative unital $\C$-algebra and $\g$ is a basic Lie superalgebra or $\fsl(n,n)$, $n \ge 2$.  We focus on this class of Lie superalgebras $\g$ since they are contragredient and always possess simple root systems satisfying a certain technical condition (see Proposition~\ref{prop:nice-system-exists}).

We begin by defining global Weyl modules for the Lie superalgebras $\g \otimes_\C A$ (Definition~\ref{def:global-Weyl}).  After giving a presentation for these modules in terms of generators and relations (Proposition~\ref{prop:gen-rel-global}), we prove that they are universal highest weight objects in a certain category (Proposition~\ref{prop:global-Weyl-universal}).  We then define local Weyl modules (Definition~\ref{def:local-Weyl}).

Next, we focus on the case where $A$ is finitely generated and the simple root system used in the definition of the Weyl module satisfies the technical condition mentioned above (condition~\eqref{eq:system-key-propery}).  Under these additional assumptions, we prove that the local Weyl modules are finite-dimensional (Theorem~\ref{thm:fin-dim-loc-weyl}), and that they also satisfy a certain univeral property with respect to so-called highest \emph{map-weight} modules (Proposition~\ref{prop:local-Weyl-universal}). Finally, we show that the local Weyl modules satisfy a nice tensor product property (Theorem~\ref{thm:tensor-property}).

The above-mentioned results demonstrate that the Weyl modules defined in the current paper satisfy many of the properties that their non-super analogues do.  However, there are some important differences.  First of all, the Borel subalgebras of basic Lie superalgebras are not all conjugate under the action of the Weyl group, in contrast to the situation for finite-dimensional simple Lie algebras.  For this reason, our definitions of Weyl modules depend on a choice of system of simple roots.  Second, the category of finite-dimensional modules for a basic Lie superalgebra is not semisimple in general, again in contrast to the non-super setting.  For this reason, the so-called \emph{Kac modules} play an important role in the representation theory.  These are maximal finite-dimensional modules of a given highest weight.  The Weyl modules defined in the current paper can be viewed as a unification of several types modules in the following sense.  If $\g$ is a simple Lie algebra, then our definitions reduce to the usual ones.  Thus, the Weyl modules defined here are generalizations of the Weyl modules in the non-super case.  On the other hand, if $A=\C$, then the global and local Weyl modules are equal and coincide with the (generalized) Kac module, which, if $\g$ is a simple Lie algebra, is the irreducible module (of a given highest weight).  These relationships can be summarized in the following diagram:

\smallskip

\begin{center}
  \begin{tikzpicture}
    \draw (0,3) node {super};
    \draw (7,3) node {non-super};
    \draw[->] (1,3) -- (5.9,3);
    \draw (-4.5,2) node {general $A$};
    \draw (-4.5,0) node {$A=\C$};
    \draw (0,2) node {global/local Weyl (super)module};
    \draw[->] (-4.5,1.6) -- (-4.5,0.4);
    \draw (7,2) node {global/local Weyl module};
    \draw (0,0) node {generalized Kac module};
    \draw (7,0) node {irreducible module};
    \draw [->,line join=round,decorate, decoration={zigzag, segment length=4,amplitude=.9, post=lineto, post length=2pt}] (3,2) -- (4.5,2);
    \draw [->,line join=round,decorate, decoration={zigzag, segment length=4,amplitude=.9, post=lineto, post length=2pt}] (2.5,0) -- (5,0);
    \draw [->,line join=round,decorate, decoration={zigzag, segment length=4,amplitude=.9, post=lineto, post length=2pt}] (0,1.6) -- (0,0.4);
    \draw [->,line join=round,decorate, decoration={zigzag, segment length=4,amplitude=.9, post=lineto, post length=2pt}] (7,1.6) -- (7,0.4);
  \end{tikzpicture}
\end{center}

The definition of global and local Weyl modules for Lie superalgebras opens up a number of directions of possible further research.  We conclude this introduction by listing some of these.
\begin{asparaenum}
  \item One should be able to define Weyl modules when $\g$ is not basic.  For example, in \cite{CMS14}, the finite-dimensional irreducible $\g \otimes A$-modules have been classified in the case that $\g$ is the queer Lie superalgebra.  The nature of the classification (in terms of evaluation modules) seems to indicate that the theory of Weyl modules should be relatively similar to the case considered in the current paper.
  \item \emph{Twisted} versions of Weyl modules have been defined and investigated in the non-super setting (see \cite{CFS08,FMS13,FKKS12,FMS14}).  One should similarly be able to develop a twisted theory of Weyl modules for equivariant map Lie superalgebras.
  \item A categorical approach to Weyl modules was developed in \cite{CFK10}.  It would be interesting to develop this theory in the super setting.  In particular, one should be able to define super analogues of Weyl functors.
  \item Recently, in \cite{SVV14,BHLW14}, local Weyl modules for current algebras have appeared as trace decategorifications of categories used to categorify quantum groups.  It is natural to ask how the super analogues of Weyl modules defined in the current paper are related to the super analogues, defined in \cite{KKT11}, of the afore-mentioned categories.
\end{asparaenum}

\medskip

\paragraph{\textbf{Acknowledgements}}  We would like to thank S.-J.~Cheng and E.~Neher for helpful conversations.

\iftoggle{detailsnote}{
\medskip

\paragraph{\textbf{Note on the arXiv version}} For the interested reader, the tex file of the \href{http://arxiv.org/abs/1505.06949}{arXiv version} of this paper includes hidden details of some straightforward computations and arguments that are omitted in the pdf file.  These details can be displayed by switching the \texttt{details} toggle to true in the tex file and recompiling.
}{}

\medskip

\paragraph{\textbf{Changes relative to published version}}  This preprint version of the paper contains a correction to the published version.  In particular, Definition~\ref{def:updated} and Remark~\ref{rem:ann-even-root} have been modified.

%
\section{Preliminaries}\label{sec:prelim}
%

In this section, we review some facts about associative commutative algebras and Lie superalgebras that will be needed in the sequel.  We also prove some results about simple root systems for basic Lie superalgebras.

\subsection{Commutative algebras}

Let $A$ denote an associative commutative unital $\C$-algebra. We define the \emph{support} of an ideal $I$ of $A$ to be
\[
  \Supp (I) = \{ \sm \in \MaxSpec A \mid I \subseteq \sm\}.
\]

\begin{lemma} \label{lem:assoc-alg-facts}
  Let $I$, $J$ be ideals of $A$.
  \begin{enumerate}
    \item \label{lem-item:support-power} For all positive integers $N$, we have $\Supp(I)=\Supp(I^N)$.
    \item \label{lem-item:finiteCod-and-finiteSupp} If $A$ is finitely generated, then the support of $I$ is finite if and only if $I$ has finite codimension in $A$.
    \item \label{lem-item:product-intersection} Suppose $\Supp(I) \cap \Supp(J) = \varnothing$.  Then $I+J=A$ and $IJ = I\cap J$.
  \end{enumerate}
\end{lemma}

\begin{proof}
  The proofs of parts \eqref{lem-item:finiteCod-and-finiteSupp} and \eqref{lem-item:product-intersection} can be found in \cite[\S 2.1]{Sav14}.  It remains to prove part \eqref{lem-item:support-power}.  Fix a positive integer $N$.  It is clear that $\Supp(I) \subseteq \Supp (I^N)$.  The reverse inclusion follows from the fact that maximal ideals are prime.
  \details{
    Suppose $\sm$ is a maximal ideal containing $I^N$.  In particular, for all $a \in I$, we have $a^N \in \sm$.  But since maximal ideals are prime ideals, this implies that $a \in \sm$.
  }
\end{proof}

\subsection{Lie superalgebras}

For a $\Z_2$-graded vector space $V_{\bar 0} \oplus V_{\bar 1}$, an element of $v \in V_i$, $i \in \Z_2$, is said to have \emph{parity} $|v| = i$.  Vectors of parity $\bar 0$ are said to be \emph{even} and those of parity $\bar 1$ are said to be \emph{odd}.

\begin{defn}[Lie superalgebra]
  A \emph{Lie superalgebra} is a $\Z_2$-graded vector space $\g=\g_{\bar 0}\oplus \g_{\bar 1}$ with a bilinear multiplication $[\cdot,\cdot]$ satisfying the following axioms:
  \begin{enumerate}
    \item The multiplication respects the grading: $[\g_i,\g_j] \subseteq \g_{i+j}$ for all $i,j \in \Z_2$.
    \item Skew-supersymmetry: $[a,b]=-(-1)^{|a||b|}[b,a]$, for all homogeneous elements $a,b\in \g$.
    \item Super Jacobi Identity: $[a,[b,c]]=[[a,b],c]+(-1)^{|a||b|}[b,[a,c]]$, for all homogeneous elements $a,b,c\in \g$.
  \end{enumerate}
\end{defn}
Observe that $\g_{\bar0}$ inherits the structure of a Lie algebra and that $\g_{\bar1}$ inherits the structure of a $\g_{\bar0}$-module.  A Lie superalgebra $\g$ is said to be \emph{simple} if there are no nonzero proper ideals, that is, there are no nonzero proper graded subspaces $\mathfrak{i} \subseteq \g$ such that $[\mathfrak{i},\g]\subseteq \mathfrak{i}$.  A finite-dimensional simple Lie superalgebra $\g = \g_{\bar0} \oplus \g_{\bar1}$ is said to be \emph{classical} if the $\g_{\bar0}$-module $\g_{\bar1}$ is completely reducible.  Otherwise, it is said to be of \emph{Cartan type}.

For a classical Lie superalgebra $\g$, the $\g_{\bar 0}$-module $\g_{\bar1}$ is either irreducible or a direct sum of two irreducible representations.  (This follows, for example, from the classification theorem \cite[Th.~2]{Kac77}, or see \cite[\S2.8]{FSS00}.)  In the first case, $\g$ is said to be of \emph{type II}, and in the second case, $\g$ is said to be of \emph{type I}. A classical Lie superalgebra is said to be \emph{basic} if it admits a nondegenerate invariant bilinear form. Otherwise, it is said to be \emph{strange}.  We will mostly be concerned with basic Lie superalgebras.  However, the majority of our results also hold for the Lie superalgebra $\fsl(n,n)$, $n \ge 2$, which is a 1-dimensional central extension of the basic Lie superalgebra $A(n-1,n-1)$.  (Throughout the paper we will somewhat abuse terminology by talking of the Lie superalgebras $A(m,n)$, $B(m,n)$, etc., instead of the Lie superalgebras of \emph{type} $A(m,n)$, $B(m,n)$, etc.)

\details{
  Table~\ref{table} lists the basic Lie superalgebras that are not Lie algebras.
  \begin{table}[!h]
    \color{OliveGreen}
    \begin{tabular}{ccc}
      \toprule
      $\g$ & $\g_{\bar 0}$ & Type \\
      \midrule
      $A(m,n)$, $m > n \ge 0$ & $A_m\oplus A_n\oplus \C$ & I \\
      $A(n,n)$, $n \ge 1$ & $A_n \oplus A_n$ & I \\
      $\fsl(n,n)$, $n \ge 2$ & $A_{n-1} \oplus A_{n-1} \oplus \C$ & N/A \\
      $C(n+1)$, $n \geq 1$ & $C_n \oplus \C$ & I \\
      $B(m,n)$, $m \geq 0$, $n\geq 1$ & $B_m\oplus C_n$ & II \\
      $D(m,n)$, $m \geq 2$, $n\geq 1$ & $D_m\oplus C_n$ & II \\
      $F(4)$ & $A_1\oplus B_3$ & II\\
      $G(3)$ & $A_1\oplus G_2$ & II \\
      $D(2,1;\alpha)$, $\alpha\neq 0,-1$ & $A_1\oplus A_1\oplus A_1$ & II \\
      \bottomrule
    \end{tabular}
    \bigskip
    \caption{The basic classical Lie superalgebras that are not Lie algebras, together with their even part and their type} \label{table}
  \end{table}
}

\subsection{Contragredient Lie superalgebras}

Let $I = \{1,\dotsc, n\}$, let $A=(a_{ij})_{i,j \in I}$ be a complex matrix, and fix a function $p \colon I \to \Z_2$.  Fix a vector space $\h$ of dimension $2n- \rank A$ and linearly independent $\alpha_i \in \h^*$, $i\in I$, and $H_i \in \h$, $i\in I$, such that $\alpha_j(H_i)=a_{ij}$, for all $i,j \in I$.  We define $\tilde{\g}(A)$ to be the Lie superalgebra generated by the even vector space $\h$ and elements $X_i, Y_i$, $i\in I$, with the parity of $X_i$ and $Y_i$ equal to $p(i)$, and subject to the relations
\begin{equation*}
  [X_i,Y_j]=\delta_{ij}H_i, \quad [H,H']=0, \quad [H,X_i]=\alpha_i(H)X_i,\quad [H,Y_i] = -\alpha_i(H)Y_i,
\end{equation*}
for $i,j \in I$ and $H,H'\in\h$.

The \emph{contragredient} Lie superalgebra $\g=\g(A)$ is defined to be the quotient of $\tilde{\g}(A)$ by the ideal that is maximal among all the ideals that intersects $\h$ trivially (see \cite[\S5.2]{Mus12}).  The images of the elements $X_i, Y_i, H_i$, $i\in I$, in $\g(A)$ are denoted by the same symbols.

Since the action of $\h$ on $\g$ is diagonalizable, we have a root space decomposition
\[
  \g=\h\oplus \bigoplus_{\alpha\in \Delta}\g_{\alpha},\quad \Delta\subseteq \h^*,
\]
where every root space $\g_\alpha$ is either purely even or purely odd. A root $\alpha$ is called \emph{even} (resp.\ \emph{odd}) if $\g_\alpha \subseteq \g_{\bar 0}$ (resp. $\g_\alpha\subseteq \g_{\bar 1}$). We denote by $\Delta_{\bar 0}$ and $\Delta_{\bar 1}$ the sets of even and odd roots respectively. A linearly independent subset $\Sigma=\{\beta_1,\dotsc,\beta_n\} \subseteq \Delta$ is called a \emph{base} if we can find $X_{\beta_i}\in\g_{\beta_i}$ and $Y_{\beta_i}\in\g_{-\beta_i}$, $i = 1,\dotsc,n$, such that $\{X_{\beta_i}, Y_{\beta_i} \mid i=1,\dotsc, n\}\cup \h$ generates $\g(A)$, and
\[
  [X_{\beta_i}, Y_{\beta_j}]=0 \text{ for } i \neq j.
\]
Defining $H_{\beta_i}=[X_{\beta_i},Y_{\beta_i}]$, it follows that the elements $X_{\beta_i}, Y_{\beta_i}$ and $H_{\beta_i}$ satisfy the following relations:
\begin{equation}\label{eq:contragredient-relations}
  [H_{\beta_j},X_{\beta_i}]=\beta_i(H_{\beta_j})X_{\beta_i},\ [H_{\beta_j},Y_{\beta_i}]=-\beta_i(H_{\beta_j})Y_{\beta_i},\ [X_{\beta_i},Y_{\beta_j}]=\delta_{i j} H_{\beta_i},\ i,j \in \{1,\dotsc,n\}.
\end{equation}
The matrix $A_\Sigma=(b_{ij})$, where $b_{ij}=\beta_j(H_{\beta_i})$, is called the \emph{Cartan matrix} with respect to the base $\Sigma$. The original set $\Pi=\{\alpha_1,\ldots, \alpha_n\}$ is called the \emph{standard base}. It is clear that $A$ is the Cartan matrix associated to $\Pi$, i.e.\ $A = A_\Pi$.  The relations \eqref{eq:contragredient-relations} imply that every root is a purely positive or purely negative integer linear combination of elements in $\Sigma$.  We call such a root positive or negative, respectively, and we have the decomposition $\Delta = \Delta^+(\Sigma) \sqcup \Delta^-(\Sigma)$, where $\Delta^+(\Sigma)$ and $\Delta^-(\Sigma)$ denote the set of positive and negative roots, respectively. A positive root is called \emph{simple} if it cannot be written as a sum of two positive roots. It is clear that a root is simple if and only if it lies in $\Sigma$.  Thus, $\Sigma$ is a \emph{system of simple roots} in the usual sense.  We define $\Sigma_z := \Sigma \cap \Delta_z$ and $\Delta_z^\pm(\Sigma) := \Delta_z \cap \Delta^\pm(\Sigma)$ for $z \in \Z_2$.  The triangular decomposition of $\g$ induced by $\Sigma$ is given by
\[
  \g=\n^-(\Sigma)\oplus \h\oplus \n^+(\Sigma),
\]
where $\n^+(\Sigma)$ (resp.\ $\n^-(\Sigma)$) is the subalgebra generated by $X_\beta$ (resp.\ $Y_\beta$), $\beta\in \Sigma$. The subalgebra $\bb(\Sigma)=\h\oplus \n^+(\Sigma)$ is called the \emph{Borel subalgebra} corresponding to $\Sigma$.  Note that $\Delta_{\bar 0}^+(\Sigma)$ is a system of positive roots for the Lie algebra $\g_{\bar 0}$. We denote by $\Sigma(\g_{\bar 0})$ the set of simple roots of $\g_{\bar 0}$ with respect to this system.

Suppose that $\g$ is equal to $A(m,n)$ with $m \neq n$, $\gl(n,n)$, $B(m,n)$, $C(n)$, $D(m,n)$, $D(2,1;\alpha)$, $F(4)$, or $G(3)$. By \cite[Theorems~5.3.2, 5.3.3 and 5.3.5]{Mus12}, we have that $\g$ is a contragredient Lie superalgebra. The Lie superalgebra $\fsl(n,n)$ (resp.\ $A(n,n)$) is isomorphic to $[\gl(n,n),\gl(n,n)]$ (resp.\ $[\gl(n,n),\gl(n,n)]/C$, where $C$ is a one-dimensional center).  The image of $X\in\fsl(n,n)$ in $A(n,n)$ will be denoted by the same symbol. Fixing a base $\Sigma$ of $\gl(n,n)$, the triangular decomposition $\gl(n,n) = \n^-(\Sigma) \oplus \h \oplus \n^-(\Sigma)$ induces the triangular decompositions
\begin{equation*}\label{eq:triang-dec-A(n,n)}
  \fsl(n,n)=\n^-(\Sigma)\oplus \h'\oplus \n^+(\Sigma) \quad \text{and} \quad A(n,n)=\n^-(\Sigma)\oplus (\h'/C)\oplus \n^+(\Sigma),
\end{equation*}
where $\h'$ is the subspace of $\h$ generated by $H_\beta$, $\beta \in \Sigma$ (see \cite[Lem.~5.2.3]{Mus12}).  In particular, any root of $\fsl(n,n)$ or $A(n,n)$ is a purely positive or a purely negative integer linear combination of elements in $\Sigma$. Therefore $\Delta=\Delta^+(\Sigma)\sqcup \Delta^-(\Sigma)$ is a decomposition of the system of roots of $\frak{sl}(n,n)$ and $A(n,n)$. The matrix $A_\Sigma$ is also called the Cartan matrix of $\frak{sl}(n,n)$ and $A(n,n)$ corresponding to $\Sigma$.

\begin{rmk} \label{rmk:dim1-root-space}
  Assume $\g$ is a basic Lie superalgebra, $\gl(n,n)$ with $n\geq 2$,  or $\fsl(n,n)$ with $n\geq 3$.  Then $[\g_\alpha,\g_\beta]\neq 0$ if $\alpha, \beta, \alpha+\beta\in \Delta$.  In particular, the parity of $\alpha + \beta$ is the sum of the parities of $\alpha$ and $\beta$.  \details{
    Since any root is purely even or purely odd, we have $\g_\gamma \subseteq\g_{|\gamma|}$, where $|\gamma|$ is the parity of $\gamma$. Thus $[\g_\alpha,\g_\beta]\subseteq [\g_{|\alpha|},\g_{|\beta|}]\subseteq \g_{|\alpha|+|\beta|}$, which implies that $\g_{\alpha+\beta}\cap \g_{|\alpha|+|\beta|}\neq 0$, and so $\g_{\alpha+\beta}\subseteq \g_{|\alpha|+|\beta|}$.
  }
  Moreover, if $\g \ne A(1,1)$, then $\dim \g_\alpha =1$ for all $\alpha\in \Delta$ (see \cite[Ch.~2]{Mus12}).
\end{rmk}

In Section~\ref{sec:local-Weyl}, we will be particularly interested in systems of simple roots $\Sigma$ satisfying the following property:
\begin{equation} \label{eq:system-key-propery}
  \text{For all } \alpha \in \Sigma_{\bar 1}, \text{ there exists } \alpha' \in \Delta_{\bar 1}^+(\Sigma) \text{ such that } \alpha + \alpha' \in \Delta(\Sigma).
\end{equation}
Note that such an element $\alpha + \alpha'$ is necessarily an even root.  Our next goal is to show that a system of simple roots satisfying \eqref{eq:system-key-propery} always exists.

Let $\Sigma$ be a system of simple roots and suppose that $\beta \in \Sigma$ is an odd root with $\beta(H_\beta)=0$. (Such a root is known as an \emph{isotropic} odd root.) Then define the \emph{reflection} $r_\beta \colon \Sigma \to \Delta$ with respect to $\beta$ by
\begin{gather*}
  r_\beta(\beta) = -\beta, \\
  r_\beta(\beta') = \beta',\quad \text{for } \beta' \in \Sigma,\ \beta' \neq \beta,\ \beta(H_{\beta'})=\beta'(H_\beta)=0, \\
  r_\beta(\beta') = \beta + \beta',\quad \text{for } \beta' \in \Sigma,\ \beta' \neq \beta,\ \beta(H_{\beta'}) \neq 0 \text{ or } \beta'(H_\beta) \neq 0.
\end{gather*}
By \cite[Lem.~1.30]{CW12}, $r_\beta(\Sigma)$ is a system of simple roots, and
\begin{equation} \label{eq:system-reflection}
  \Delta^+(r_\beta(\Sigma)) \setminus \{-\beta\} = \Delta^+(\Sigma) \setminus \{\beta\}.
\end{equation}
(We use here the fact that $\gl(n,n)$ and $\fsl(n,n)$ have the same system of simple roots as $A(n,n)$.)

If $\g$ is a basic Lie superalgebra, $\gl(n,n)$, $n \ge 2$, or $\fsl(n,n)$, $n\geq 2$, then $\g$ admits a system of simple roots with only one odd root (see \cite[Tables~3.4.4 and~5.3.1]{Mus12}).  Let $\Pi = \{\gamma_1,\dotsc,\gamma_n\}$ denote such a system, and let $s \in \{1,\dotsc,n\}$ such that $\gamma_s$ is the unique odd root that lies in $\Pi$.  The system $\Pi$ is often called a \emph{distinguished} system of simple roots.  When $\g \neq B(0,n)$, we have that $\gamma_s$ is an odd isotropic regular root. Then we can consider the odd reflection $r_{\gamma_s}$ with respect to $\gamma_s$.

\begin{prop} \label{prop:nice-system-exists}
  Let $\Pi$ be a distinguished system of simple roots for $\g$.
  \begin{enumerate}
    \item \label{lem-item:key-system-type-II} If $\g$ is a basic Lie superalgebra of type II, then $\Pi$ satisfies condition \eqref{eq:system-key-propery}.

    \item \label{lem-item:key-system-reflected} If $\g$ is $\gl(n,n)$, $n \ge 2$, $\fsl(n,n)$, $n \ge 2$, or a basic Lie superalgebra other than $B(0,n)$, then $r_{\gamma_s}(\Pi)$ satisfies condition \eqref{eq:system-key-propery}.
  \end{enumerate}
  In particular, if $\g$ is a basic Lie superalgebra, $\gl(n,n)$, $n \ge 2$, or $\fsl(n,n)$, $n\geq 2$, then it admits at least one system of simple roots satisfying \eqref{eq:system-key-propery}.
\end{prop}

\begin{proof}
  Part \eqref{lem-item:key-system-type-II} follows from direct examination of the distinguished root systems in type II (see, for example, \cite[Tables~3.54, 3.57--3.60]{FSS00}).
  \details{
    In a distinguished root system, there is only one simple odd root.  We list below this simple odd root $\gamma \in \Sigma_{\bar 1}$, together with an element $\gamma' \in \Delta_{\bar 1}^+$ such that $\alpha + \alpha' \in \Delta$.  Notation is as in \cite[Tables~3.54, 3.57--3.60]{FSS00}.
    \begin{itemize}
      \item $\g = B(m,n)$, $m \geq 0$, $n\geq 1$; $\gamma = \alpha_n$; $\gamma' = \alpha_n + 2 \alpha_{n+1} + \dotsb + 2 \alpha_{n+m}$
      \item $\g = D(m,n)$, $m \geq 2$, $n\geq 1$; $\gamma = \alpha_n$; $\gamma' = \alpha_n + 2 \alpha_{n+1} + \dotsb + 2 \alpha_{n+m-2} + \alpha_{n+m-1} + \alpha_{n+m}$
      \item $\g = F(4)$; $\gamma = \alpha_1$; $\gamma' = \alpha_1 + 3 \alpha_2 + 2 \alpha_3 + \alpha_4$
      \item $\g = G(3)$; $\gamma = \alpha_1$; $\gamma' = \alpha_1 + 4 \alpha_2 + 2 \alpha_3$
      \item $\g = D(2,1;\alpha)$, $\alpha\neq 0,-1$; $\gamma = \alpha_1$; $\gamma' = \alpha_1 + \alpha_2 + \alpha_3$
    \end{itemize}
    One sees that, in each case, $\gamma + \gamma' \in \Delta$.
  }

  Now suppose that $\gamma_s$ is isotropic and let $\Pi' = r_{\gamma_s}(\Pi)$.  To prove part \eqref{lem-item:key-system-reflected}, we will show that $\alpha + r_{\gamma_s}(\gamma_s)\in \Delta_{\bar 0}^+(\Pi')$, for all odd roots $\alpha\in \Pi'\setminus\{r_{\gamma_s}(\gamma_s)\}$.  First assume that $\g$ is $\gl(n,n)$ ($n \ge 2$), $\fsl(n,n)$ ($n \ge 2$), or a basic Lie superalgebra other than $B(0,n)$ or $D(2,1;\alpha)$.  One can verify, by looking at each distinguished  Cartan matrix, that $\gamma_s(H_{\gamma_{s\pm 1}})=-1$ (when $1 \le s \pm 1 \le n$) and $\gamma_s(H_{\gamma_{s\pm j}})=0$ when $j\geq 2$ (and $1 \le s \pm j \le n$).  (See, for example, \cite[Tables~3.53--3.58]{FSS00}.  The odd root $\gamma_s$ is indicated there by an X on the corresponding node in the Dynkin diagram.)  Thus
  \[
    r_{\gamma_s}(\gamma_s)=-\gamma_s,\quad r_{\gamma_s}(\gamma_{s\pm 1})=\gamma_s+\gamma_{s\pm 1}\text{ and } r_{\gamma_s}(\gamma_{s\pm j})=\gamma_{s\pm j},\text{ for all }j\geq 2.
  \]
  Since the only odd root in $\Pi$ is $\gamma_s$, the odd roots of $\Pi'$ are precisely $r_{\gamma_s}(\gamma_{s-1}),r_{\gamma_s}(\gamma_s), r_{\gamma_s}(\gamma_{s+1})$.  Now, by \eqref{eq:system-reflection}, we have $\Delta^+(\Pi)\setminus\{\gamma_s\}=\Delta^+(\Pi')\setminus\{r_{\gamma_s}(\gamma_s)\}$,
  which implies that $\Delta_{\bar 0}^+(\Pi)=\Delta_{\bar 0}^+(\Pi')$. Thus
  \[
    r_{\gamma_s}(\gamma_{s\pm 1}) + r_{\gamma_s}(\gamma_s) = \gamma_{s\pm 1} \in \Delta_{\bar 0}^{+}(\Pi)=\Delta_{\bar 0}^{+}(\Pi').
  \]

  Finally, assume $\g=D(2,1;\alpha)$. Then $\Pi=\{\gamma_1,\gamma_2,\gamma_3\}$, where $s=1$ and $\gamma_1(H_{\gamma_{j}})=-1$, for $j=2,3$ (see \cite[Table~3.60]{FSS00}).  Then every element of $\Pi'=\{r_{\gamma_1}(\gamma_1),r_{\gamma_1}(\gamma_2),r_{\gamma_1}(\gamma_3)\}$ is odd, and again $r_{\gamma_1}(\gamma_j) + r_{\gamma_1}(\gamma_1) = \gamma_j \in \Delta_{\bar 0}^+(\Pi) = \Delta_{\bar 0}^+(\Pi')$, for $j=2,3$.
\end{proof}

\begin{rmk}
  There exist systems of simple roots that do not satisfy \eqref{eq:system-key-propery}.  For instance, if $\g$ is of type I, then a distinguished system of simple roots does not satisfy \eqref{eq:system-key-propery}.  This follows from the fact that the induced $\Z$-gradation is of the form $\g_{-1} \oplus \g_0 \oplus \g_1$ (see \cite[Prop.~1.6]{Kac78}).
\end{rmk}

\subsection{Generalized Kac modules}

For the remainder of the paper, we assume that $\g$ is a basic Lie superalgebra or $\fsl(n,n)$, $n\geq 2$.  We fix a system of simple roots $\Sigma$, define
\[
  \Delta_z^+=\Delta_z^+(\Sigma) \text{ for all } z \in \Z_2,
\]
and let $\g = \n^-\oplus \h\oplus \n^+$ be the triangular decomposition induced by $\Sigma$, i.e.\ $\n^\pm=\n^\pm(\Sigma)$.  In the case that $\g$ is $\fsl(n,n)$ or $A(n,n)$, we consider the triangular decomposition induced by $\gl(n,n)$.  Recall that the elements $X_\alpha$, $Y_\alpha$, $\alpha \in \Sigma$, generate the subalgebras $\n^+$ and $\n^-$, respectively.

Since $\g_{\bar 0}$ is a reductive Lie algebra, for each even root $\alpha$ we can choose elements $X_\alpha \in \g_{\alpha}, Y_\alpha \in \g_{-\alpha}$, and $H_\alpha \in \h$, such that the subalgebra generated by these elements is isomorphic to $\fsl(2)$, with these elements satisfying the relations for the standard Chevalley generators.  In this case, we say the set $\{X_\alpha, Y_\alpha, H_\alpha\}$ is an $\fsl(2)$-triple.

We denote the irreducible highest weight $\g$-module with highest weight $\lambda \in \h^*$ by $V(\lambda)$. Define
\begin{equation}
  \Lambda^+ = \Lambda^+(\Sigma) = \{\lambda\in\h^* \mid \dim V(\lambda)<\infty\}.
\end{equation}
Note that, for $\lambda \in \Lambda^+$, since $V(\lambda)$ is finite dimensional, we have $\lambda(H_\alpha)\in \N$, for all $\alpha \in \Sigma(\g_{\bar 0})$.
\details{
  This follows from considering $V(\lambda)$ as a $\g_{\bar 0}$-module.
}

\begin{defn}[The module $\bar V(\lambda)$] \label{def-kac-mod}
  For $\lambda\in \Lambda^+$, we define ${\bar V}(\lambda)$ to be the $\g$-module generated by a vector $v_\lambda$ with defining relations
  \begin{equation} \label{kac-mod-relations}
    \n^+ v_\lambda=0,\quad hv_\lambda=\lambda(h)v_\lambda,\quad Y_\alpha^{\lambda(H_\alpha)+1}v_\lambda=0,\quad \text{for all } h \in \h,\ \alpha \in \Sigma(\g_{\bar 0}).
  \end{equation}
\end{defn}

\begin{prop} \label{prop:Vbar-fd}
  For all $\lambda \in \Lambda^+$, the module ${\bar V}(\lambda)$ is finite-dimensional.
\end{prop}

\begin{proof}
  Let $L(\lambda)$ be the irreducible $\g_{\bar 0}$-module of highest weight $\lambda$.  Since $\g_{\bar 0}$ is a reductive Lie algebra and $\lambda(H_\alpha) \in \N$, for all $\alpha \in \Sigma(\g_{\bar 0})$, we have that $L(\lambda)$ is finite dimensional. Moreover, it is well known that $L(\lambda)$ is isomorphic to the $\g_{\bar 0}$-module generated a vector $u_\lambda$ with defining relations
  \[
    \n_{\bar 0}^+ u_\lambda=0,\quad h u_\lambda = \lambda(h) u_\lambda,\quad Y_\alpha^{\lambda(H_\alpha)+1} u_\lambda=0,\quad \text{for all } h \in \h,\ \alpha \in \Sigma(\g_{\bar 0}).
  \]
  Let $V'=U(\g_{\bar 0})v_\lambda\subseteq {\bar V}(\lambda)$ be the $\g_{\bar 0}$-submodule of ${\bar V}(\lambda)$ generated by $v_\lambda$.  Then the map given by
  \[
    \varphi \colon L(\lambda)\rightarrow V',\quad  x u_\lambda\mapsto x v_\lambda,\quad \text{for all } x \in U(\g_{\bar 0}),
  \]
  is a well-defined epimorphism of $\g_{\bar 0}$-modules.
  \details{
    Indeed, since
    \begin{gather*}
      \varphi((h-\lambda(h)) u_\lambda)=(h-\lambda(h)) v_\lambda=0, \text{ for all }h\in \h,\\
      \varphi(Y_\alpha^{\lambda(H_\alpha)+1} u_\lambda) = Y_\alpha^{\lambda(H_\alpha)+1} v_\lambda=0, \text{ for all } \alpha \in \Sigma(\g_{\bar 0}),\\
      \varphi(\n_{\bar 0}^+ u_\lambda)=\n_{\bar 0}^+v_\lambda\subseteq \n^+v_\lambda=0,
    \end{gather*}
    we have that
    \[
      (x-x') u_\lambda=0 \implies \varphi ((x-x')u_\lambda)=0, \text{ for all }x,x'\in U(\g_{\bar 0}).
    \]
    But this implies that
    \[
      x u_\lambda = x' u_\lambda \implies \varphi(x u_\lambda)=\varphi(x' u_\lambda), \text{ for all } x,x'\in U(\g_{\bar 0}).
    \]
    Therefore, the map is well defined.
  }
  Thus, $V'$ is finite dimensional.  Then it follows from the PBW Theorem for Lie superalgebras (see, for instance, \cite[Th.~1.36]{CW12}) that $\bar V(\lambda)$ is finite dimensional.
  \details{
    Let $\{x_i \mid 1\leq i\leq \dim \g_{\bar 1}\}$ be a basis of $\g_{\bar 1}$.  Then the PBW Theorem implies that
    \[
      {\bar V}(\lambda)=\sum_{1\leq j_1< \dotsb <j_k\leq \dim \g_{\bar 1}} x_{j_1} \dotsm x_{j_k} V'.
    \]
  }
\end{proof}

\begin{lemma} \label{univ-prop-kac-mod}
  Suppose $V$ is a finite-dimensional $\g$-module generated by a highest weight vector of weight $\lambda\in\Lambda^+$.  Then there exists an unique submodule $W$ of ${\bar V}(\lambda)$ such that ${\bar V}(\lambda)/W \cong V$ as $\g$-modules.
\end{lemma}

\begin{proof}
  Let $v \in V_\lambda$ be a highest weight vector.  Then the first two relations in \eqref{kac-mod-relations} are satisfied by $v$, by the definition of a highest weight vector.  The fact that $\g_{\bar 0}$ is a reductive Lie algebra and $V$ is finite dimensional implies that $v$ also satisfies the last relation in \eqref{kac-mod-relations}.  Thus the map ${\bar V}(\lambda)\rightarrow V$ defined by extending the assignment $v_\lambda\mapsto v$ is a well-defined epimorphism of $\g$-modules. Since $\dim V_\lambda = 1 = \dim {\bar V}(\lambda)_\lambda$ and homomorphisms between modules preserve weight spaces, this map is unique up to scalar multiple.  Thus, the kernel $W$ of this map is unique.
\end{proof}

Since every irreducible finite-dimensional $\g$-module is generated by a highest weight vector of weight $\lambda\in \Lambda^+$,  Lemma~\ref{univ-prop-kac-mod} applies to irreducible finite-dimensional $\g$-modules.

\begin{rmk} \label{rmk:gen-Kac}
  It follows from Lemma~\ref{univ-prop-kac-mod} that $\bar V(\lambda)$ coincides with the \emph{generalized Kac module} defined in \cite[p.~689]{Cou14}.  Thus, when $\Sigma$ is a distinguished root system, it follows from \cite[Lem.~11]{Cou14} that $\bar V(\lambda)$ is isomorphic to the usual Kac module defined in \cite[p.~613]{Kac78}.
\end{rmk}

\section{Global Weyl modules}

Recall that $\g$ is either a basic classical Lie superalgebra or $\fsl(n,n)$, $n \ge 2$.  Let $A$ be an associative commutative unital $\C$-algebra.  We can then consider the Lie superalgebra $\g \otimes_\C A$, where the $\Z_2$-grading is given by $(\g\otimes A)_j=\g_j\otimes A, j\in\Z_{2}$, and the bracket is determined by $[x_{1}\otimes a_{1},x_{2}\otimes a_{2}]=[x_{1},x_{2}]\otimes a_{1}a_{2}$ for $x_{i}\in\g$, $a_{i}\in A$, $i\in \{1,2\}$. We refer to a superalgebra of this form as a \emph{map Lie superalgebra}, inspired by the case where $A$ is the ring of regular functions on an algebraic variety.  From now on, we consider $\g \subseteq \g \otimes A$ as a subalgebra via the natural isomorphism $\g \cong \g \otimes \C$.

Let $\I$ be the full subcategory of the category of $\g_{\bar 0}$-modules whose objects are those modules that are isomorphic to direct sums of irreducible finite-dimensional $\g_{\bar 0}$-modules. Note that, if $V\in\I$, then every element of $V$ lies in a finite-dimensional $\g_{\bar 0}$-submodule of $V$.  Let $\I(\g\otimes A,\g_{\bar 0})$ denote the full subcategory of the category of $\g \otimes A$-modules whose objects are the $\g \otimes A$-modules whose restriction to $\g_{\bar 0}$ lies in $\I$.

If $V$ is a $\g$-module, then, by the PBW Theorem, we have an isomorphism of vector spaces
\begin{equation}\label{cha-fou-equ.1}
  P_A(V) := U(\g\otimes A)\otimes_{U(\g)} V\cong U(\g\otimes A_+)\otimes_\C V,
\end{equation}
where $A_+$ is a vector space complement to $\C \subseteq A$.  We will view $V$ as a $\g$-submodule of $P_A(V)$ via the natural identification $V \cong \C \otimes V \subseteq P_A(V)$.

\begin{lemma}\label{Fou14-lem3.4}
Let $V$ be a $\g$-module whose restriction to $\g_{\bar 0}$ lies in $\I$. Then $P_A(V)\in \I(\g\otimes A,\g_{\bar 0})$.
\end{lemma}
\begin{proof}
The proof is the same of that in \cite[Lem.~3.4]{FMS14}, where $\g$ is a finite-dimensional simple Lie algebra.
\end{proof}

\begin{prop}\label{cha-fou-pro3}
  If $\lambda\in \Lambda^+$, then $P_A({\bar V}(\lambda))$ is generated, as a $U(\g\otimes A)$-module, by the element $v_\lambda$, with defining relations
  \begin{equation}\label{cha-fou-equ.2}
    \n^+ v_\lambda=0,\quad h v_\lambda=\lambda(h) v_\lambda,\quad Y_\alpha^{\lambda(H_\alpha)+1} v_\lambda=0,\quad \text{for all } h \in \h,\ \alpha \in \Sigma(\g_{\bar 0}).
  \end{equation}
\end{prop}

\begin{proof}
  It is obvious that the element $v_\lambda\in P_A({\bar V}(\lambda))$ satisfies the relations \eqref{cha-fou-equ.2}. To check that these are all the relations, let $W$ be the $\g\otimes A$-module generated by a vector $w$ with defining relations \eqref{cha-fou-equ.2}.  Then we have a surjective homomorphism of $\g \otimes A$-modules $\pi_1 \colon W \to P({\bar V}(\lambda))$ which maps $w$ to $v_\lambda$.  Now, by relations \eqref{cha-fou-equ.2}, $w\in W$ generates a $\g$-submodule of $W$ isomorphic to ${\bar V}(\lambda)$. Thus, we have an epimorphism
  \[
    \pi_2: P({\bar V}(\lambda))\rightarrow W, \quad u_1\otimes_{U(\g)}u_2 v_\lambda\mapsto u_1u_2w,\quad  u_1\in U(\g\otimes A), u_2 \in U(\g).
  \]
  \details{
    This follows from the fact that $\{u\in U(\g\otimes A) \mid u(1\otimes_{U(\g)} v_\lambda)=0\}=\{u\in U(\g) \mid 1 \otimes_{U(\g)}u v_\lambda=0\}$.
  }
  Since $\pi_1=\pi_2^{-1}$, we have $W\cong P({\bar V}(\lambda))$.
\end{proof}

For $\nu\in \Lambda^+$ and $V\in \I(\g\otimes A,\g_{\bar 0})$, let $V^\nu$ be the unique maximal $\g\otimes A$-module quotient of $V$ such that the weights of $V^\nu$ lie in $\nu-Q^+$, where $Q^+=\sum_{\alpha \in \Sigma} \N \alpha$ is the positive root lattice of $\g$.  In other words,
\[
  V^\nu=V/\sum_{\mu \not \in \nu - Q^+}U(\g\otimes A)V_\mu.
\]

Note that a morphism $\varphi \colon V \to W$ of objects in $\I (\g\otimes A, \g_{\bar 0})$ induces a morphism $\varphi^\nu \colon V^\nu \to W^\nu$.  Let $\I(\g\otimes A,\g_{\bar 0})^\nu$ denote the full subcategory of $\I(\g\otimes A,\g_{\bar 0})$ whose objects are those $V \in \I(\g\otimes A,\g_{\bar 0})$ such that $V^\nu=V$.  Proposition~\ref{prop:Vbar-fd} and Lemma~\ref{Fou14-lem3.4} imply that $P_A({\bar V}(\lambda))\in \I(\g\otimes A,\g_{\bar 0})$ for all $\lambda \in \Lambda^+$.

\begin{defn}[Global Weyl module] \label{def:global-Weyl}
  We define the \emph{global Weyl module} associated to $\lambda\in \Lambda^+$ to be
  \[
   W(\lambda):=P_A({\bar V}(\lambda))^\lambda.
  \]
  We let $w_\lambda$ denote the image of $v_\lambda$ in $W(\lambda)$.
\end{defn}

\begin{prop} \label{prop:gen-rel-global}
  For $\lambda\in \Lambda^+$, the global Weyl module $W(\lambda)$ is generated by $w_\lambda$, with defining relations
  \begin{equation}\label{cha-fou-equ5}
    (\n^+\otimes A)w_\lambda=0,\quad hw_\lambda=\lambda(h)w_\lambda,\quad Y_\alpha^{\lambda(H_\alpha)+1}w_\lambda=0,\quad \text{for all } h \in \h,\ \alpha \in \Sigma(\g_{\bar 0}).
  \end{equation}
\end{prop}

\begin{proof}
  Since the weights of $W(\lambda)$ lie in $\lambda-Q^+$, it follows that $(\n^+\otimes A)w_\lambda=0$. The remaining relations are clear since they are already satisfied by $v_\lambda$. To prove that these are the only relations, let $W$ be the module generated by an element $w$ with relations \eqref{cha-fou-equ5}, so that we have an epimorphism $\pi_1 \colon W \twoheadrightarrow W(\lambda)$ sending $w$ to $w_\lambda$.  Since the relations \eqref{cha-fou-equ5} imply the relations \eqref{kac-mod-relations}, the vector $w \in W$ generates a $\g$-submodule of $W$ isomorphic to a quotient of $\bar V(\lambda)$.  Thus we have a surjective homomorphism
  \[
    \pi_2 \colon P_A(\bar V(\lambda)) \to W,\quad u_1 \otimes_{U(\g)} u_2 v_\lambda \mapsto u_1 u_2 w,\quad u_1 \in U(\g \otimes A),\ u_2 \in U(\g).
  \]
  Since the $\g$-weights of $W$ are bounded above by $\lambda$, it follows that $\pi_2$ induces a map $W(\lambda) \to W$ inverse to $\pi_1$.
\end{proof}

In the non-super setting, Proposition~\ref{prop:gen-rel-global} was proved in \cite[Prop.~4]{CFK10}.

\begin{prop} \label{prop:global-Weyl-universal}
  The global Weyl module $W(\lambda)$ is the unique object of $\I(\g\otimes A,\g_{\bar 0})$, up to isomorphism, that is generated by a highest weight vector of weight $\lambda$ and admits a surjective homomorphism to any object of $\I(\g\otimes A,\g_{\bar 0})$ also generated by a highest weight vector of weight $\lambda$.
\end{prop}

\begin{proof}
  Let $V\in \I(\g\otimes A,\g_{\bar 0})$ be generated by a highest weight vector $v$ of weight $\lambda$. Then
  \[
    (\n^+\otimes A)v=0,\quad hv=\lambda(h)v,\quad \text{for all } h\in\h.
  \]
  Since the $\g_{\bar 0}$-module generated by $v$ is finite-dimensional, we have that $Y_\alpha^{\lambda(H_\alpha)+1}v=0$ for all $\alpha \in \Sigma(\g_{\bar 0})$. Thus, by Proposition~\ref{prop:gen-rel-global}, we have a surjective homomorphism $W(\lambda) \twoheadrightarrow V$ such that $w_\lambda \mapsto  v$.

  Suppose that $W$ is another object of $\I(\g\otimes A,\g_{\bar 0})$ that is generated by a highest weight vector $w$ of weight $\lambda$ and admits a surjective homomorphism to any object of $\I(\g\otimes A,\g_{\bar 0})$ also generated by a highest weight vector of weight $\lambda$.  In particular, we have a surjective homomorphism $\pi_1 \colon W \twoheadrightarrow W(\lambda)$.  It follows from the PBW Theorem that $W(\lambda)_\lambda = U(\h \otimes A_+) \otimes_\C w_\lambda$.  The only elements of this weight space that generate $W(\lambda)$ are the $\C$-multiples of $w_\lambda$.  Thus, possibly after rescaling, we have $\pi_1(w)=w_\lambda$.  Now, as above, $w$ satisfies the relations \eqref{cha-fou-equ5}.  Thus there exists a homomorphism $\pi_2 \colon W(\lambda) \to W$ sending $w_\lambda$ to $w$.  It follows that $\pi_1$ and $\pi_2$ are mutually inverse homomorphisms, and so $W \cong W(\lambda)$.
\end{proof}

Note that, when $A = \C$, the global Weyl module $W(\lambda)$ coincides with the generalized Kac module $\bar V(\lambda)$.  In this case, Proposition~\ref{prop:global-Weyl-universal} reduces to the universal property given in Lemma~\ref{univ-prop-kac-mod}.

\section{Local Weyl modules} \label{sec:local-Weyl}

Recall that $\g$ is either a basic classical Lie superalgebra or $\fsl(n,n)$, $n \ge 2$, and that $A$ is an associative commutative unital $\C$-algebra.  The aim now is to describe, in terms of generators and relations, a universal object in the full subcategory of $\I(\g\otimes A,\g_{\bar 0})$ whose objects are the finite-dimensional modules generated by a highest map-weight vector of a fixed highest map-weight (see Definition~\ref{def:heighest-map-wieght}).

\begin{defn}[Local Weyl module] \label{def:local-Weyl}
  Let $\psi\in (\h\otimes A)^*$ such that $\lambda = \psi |_{\h} \in \Lambda^+$.  We define the \emph{local Weyl module} $W(\psi)$ associated to $\psi$ to be the $\g\otimes A$-module generated by a vector $w_\psi$ with defining relations
  \begin{equation}\label{relations-local}
    (\n^+\otimes A)w_\psi=0,\quad x w_\psi=\psi(x) w_\psi,\quad Y_\alpha^{\lambda(H_\alpha)+1} w_\psi=0,\quad \text{for all } x \in \h \otimes A,\ \alpha \in \Sigma(\g_{\bar 0}).
  \end{equation}
\end{defn}

\begin{defn}[Highest map-weight module] \label{def:heighest-map-wieght}
  A $\g\otimes A$-module generated by a vector $w_\psi$ satisfying the first and second relations of \eqref{relations-local} is called a \emph{highest map-weight module} with \emph{highest map-weight} $\psi$. The vector $w_\psi$ is called a \emph{highest map-weight vector} of \emph{map-weight} $\psi$.
\end{defn}

Recall that, for each $\alpha \in \Delta_{\bar 0}^+$, we have an $\fsl(2)$-triple $\{X_\alpha, Y_\alpha, H_\alpha\}$.

\begin{lemma}
  Suppose $\psi \in (\h \otimes A)^*$ such that $\lambda = \psi |_\h \in \Lambda^+$.  If $\alpha \in \Delta_{\bar 0}^+$, then $Y_\alpha^{\lambda(H_\alpha)+1} w_\psi=0$.
\end{lemma}

\begin{proof}
  The vector $Y_\alpha^{\lambda(H_\alpha)+1} w_\psi$ has weight $\lambda - (\lambda(H_\alpha)+1)\alpha$. On the other hand, it follows from Proposition~\ref{prop:gen-rel-global} that $W(\psi)$ is a quotient of the global Weyl module $W(\lambda)$, and so it is a direct sum of irreducible finite-dimensional $\g_{\bar 0}$-modules. This implies that the weights of $W(\lambda)$ are invariant under the action of the Weyl group of $\g_{\bar 0}$. But, if $s_\alpha$ denotes the reflection associated to the root $\alpha$, then $s_\alpha (\lambda - (\lambda(H_\alpha)+1)\alpha)=\lambda+\alpha$ does not lie below $\lambda$.  Therefore, $Y_\alpha^{\lambda(H_\alpha)+1} w_\psi=0$.
\end{proof}

Let $u$ be an indeterminate and, for $a\in A$, $\alpha \in \Delta_{\bar 0}^+$, define the following power series with coefficients in $U(\h\otimes A)$:
\begin{equation}\label{def-polis}
  p(a,\alpha) = \exp \left( -\sum_{i=1}^\infty \frac{H_\alpha\otimes a^i}{i} u^i\right).
\end{equation}
For $i \in \N$, let $p(a,\alpha)_i$ denote the coefficient of $u^i$ in $p(a,\alpha)$. In particular, $p(a,\alpha)_0=1$.

\begin{lemma}\label{lemm5-chari}
  Suppose $m \in \N$, $a\in A$, and $\alpha\in \Delta_{\bar 0}^+$.  Then
  \begin{equation}\label{lemm5-chari-eq1}
    (X_\alpha\otimes a)^m (Y_\alpha\otimes 1)^{m+1} - (-1)^m \sum_{i=0}^m (Y_\alpha\otimes a^{m-i}) p(a,\alpha)_i \in U(\g\otimes A)(\n^+ \otimes A).
  \end{equation}
\end{lemma}

\begin{proof}
  This formula is proved in \cite[Lem.~1.3(ii)]{CP01} in the case that $A$ is $\C[t^{\pm1}]$.
  \details{
    The translation between the notation of \cite{CP01} and the notation we use here is given by
    \[
      x_{\alpha,k}^+ \leftrightarrow X_\alpha\otimes t^k,\quad x_{\alpha,k}^- \leftrightarrow Y_\alpha\otimes t^k,\quad h_{\alpha,k} \leftrightarrow H_\alpha \otimes t^k,\quad \text{ and } \Lambda_{\alpha,k} \leftrightarrow p(t,\alpha)_k,\quad \alpha \in \Delta_{\bar 0}^+,\ k \in \N.
    \]
    Furthermore, we also have that $U \leftrightarrow U(\g\otimes A)$ and $U(>)_+$ denotes the augmentation ideal of $U(\n^+\otimes A)$. Thus $U(>)_+\subseteq U(\g\otimes A) (\n^+\otimes A)$.
  }
  However, since the fact that $t$ is an invertible element in $\C[t^{\pm 1}]$ is not used in that proof, the result is still true when $A$ is equal to $\C[t]$. Now, applying the Lie algebra homomorphism
  \[
    \fsl(2)\otimes \C[t]\rightarrow \fsl(2)\otimes A,\quad x\otimes t^m \mapsto x\otimes a^m,\quad m\in\N,\ x\in\fsl(2),
  \]
  gives our result.
\end{proof}

For the remainder of the paper we assume that
\begin{center}
  $A$ is finitely generated.
\end{center}

\begin{prop} \label{fin-dim-even}
  Suppose $\psi \in (\h \otimes A)^*$ such that $\lambda = \psi |_\h \in \Lambda^+$.  If $\alpha \in \Delta_{\bar 0}^+$, $a_1,a_2,\dotsc,a_t \in A$, and $m_1,\dotsc,m_t \in \N$, then
  \begin{equation} \label{eq:falpha-span-inclusion}
    (Y_\alpha\otimes a_1^{m_1} \dotsm a_t^{m_t}) w_\psi \in \vspan_\C\ \{(Y_\alpha\otimes a_1^{\ell_1}\dotsm a_t^{\ell_t})w_\psi \mid 0\leq \ell_i <\lambda(H_\alpha),\ i=1,\dotsc, t \}.
  \end{equation}
  In particular, $(Y_\alpha\otimes A) w_\psi$ is finite dimensional.
\end{prop}

\begin{proof}
  From the first and third relations in \eqref{relations-local}, together with \eqref{lemm5-chari-eq1}, it follows that, for $a \in A$ and $m \ge \lambda(H_\alpha)$, we have
  \[
    0=(X_\alpha\otimes a)^m(Y_\alpha\otimes 1)^{m+1} w_\psi = \sum_{i=0}^m (-1)^m (Y_\alpha\otimes a^{m-i})p(a,\alpha)_i w_\psi,
  \]
  for any $a\in A$. Since $p(a,\alpha)_0=1$, we have
  \[
    (Y_\alpha \otimes a^m)w_\psi\in \vspan_\C\{(Y_\alpha\otimes a^\ell)w_\psi \mid 0\leq\ell < m\}.
  \]
  This implies, by induction, that
  \begin{equation}\label{eq5.4-fou-sav}
    (Y_\alpha\otimes a^{m})w_\psi\in \vspan_\C\{(Y_\alpha\otimes a^\ell)w_\psi \mid 0\leq\ell <\lambda(H_\alpha)\},\quad \text{for all }m\in \N,\,a\in A.
  \end{equation}

  We will now prove \eqref{eq:falpha-span-inclusion} by induction on $t$.  The case $t=1$ follows immediately from \eqref{eq5.4-fou-sav}.  Assume that \eqref{eq:falpha-span-inclusion} holds for some $t \ge 1$.  Let $m_1,\dotsc,m_{t+1} \in \N$ and choose $h \in \h$ such that $\alpha(h) \ne 0$.  Then
  \[
    (h\otimes a_{t+1}^{m_{t+1}})(Y_\alpha \otimes a_1^{m_1} \dotsm a_t^{m_t} )w_\psi = \big( -\alpha(h)(Y_\alpha\otimes a_1^{m_1} \dotsm a_{t+1}^{m_{t+1}}) +(Y_\alpha \otimes a_1^{m_1} \dotsm a_t^{m_t})(h\otimes a_{t+1}^{m_{t+1}}) \big) w_\psi,
  \]
  and so
  \begin{equation} \label{eq5.5-fou-sav}
    (h\otimes a_{t+1}^{m_{t+1}})(Y_\alpha \otimes a_1^{m_1} \dotsm a_t^{m_t} )w_\psi + \alpha(h)(Y_\alpha\otimes a_1^{m_1} \dotsm a_{t+1}^{m_{t+1}}) w_\psi \in \vspan_\C\{(Y_\alpha\otimes a_1^{m_1} \dotsm a_t^{m_t} )w_\psi\},
  \end{equation}
  since $(h\otimes a_{t+1}^{m_{t+1}}) w_\psi \in \C w_\psi$.  By the inductive hypothesis, we have
  \[
    (Y_\alpha \otimes a_1^{m_1} \dotsm a_{t+1}^{m_{t+1}}) w_\psi \in \vspan_\C \left\{ (h \otimes a_{t+1}^{m_{t+1}}) (Y_\alpha \otimes a_1^{\ell_1} \dotsm a_t^{\ell_t}) w_\psi, (Y_\alpha \otimes a_1^{\ell_1} \dotsm a_t^{\ell_t}) w_\psi \mid 0 \le \ell_i < \lambda(H_\alpha) \right\}.
  \]
  Then, by \eqref{eq5.5-fou-sav} (with $m_i = \ell_i$ for $i=1,\dotsc,t$), we have
  \[
    (Y_\alpha \otimes a_1^{m_1} \dotsm a_{t+1}^{m_{t+1}}) w_\psi \in \vspan_\C \left\{ (Y_\alpha \otimes a_1^{\ell_1} \dotsm a_t^{\ell_t} a_{t+1}^{m_{t+1}}) w_\psi, (Y_\alpha \otimes a_1^{\ell_1} \dotsm a_t^{\ell_t}) w_\psi \mid 0 \le \ell_i < \lambda(H_\alpha) \right\}.
  \]
  Since the above inclusion holds for all $m_1,\dotsc,m_{t+1} \in \N$, we can interchange the roles of $m_1$ and $m_{t+1}$ to obtain
  \[
    (Y_\alpha \otimes a_1^{m_1} \dotsm a_{t+1}^{m_{t+1}}) w_\psi \in \vspan_\C \left\{ (Y_\alpha \otimes a_1^{\ell_1} \dotsm a_t^{\ell_t} a_{t+1}^{\ell_{t+1}}) w_\psi \mid 0 \le \ell_i < \lambda(H_\alpha) \right\}.
  \]
  This completes the proof of the inductive step.  The final statement of the lemma follows from the fact that $A$ is finitely generated.
\end{proof}

Let
\[
  \cL (\h\otimes A)=\{\psi\in (\h\otimes A)^* \mid \psi(\h\otimes I)=0, \text{ for some finite-codimensional ideal $I$ of $A$}\}.
\]

\begin{prop}\label{trivial-weyl}
  Suppose $\psi \in (\h \otimes A)^*$ such that $\lambda = \psi |_\h \in \Lambda^+$.   If $\psi \not \in \mathcal{L}(\mathfrak{h} \otimes A)$, then $W(\psi)=0$.
\end{prop}

\begin{proof}
  Let $\alpha \in \Delta_{\bar 0}^+$ and let $I_{\alpha}$ be the kernel of the linear map
  \begin{gather*}
    A \to \Hom_\C \left( W(\psi)_\lambda \otimes \g_{-\alpha},(\g_{-\alpha}\otimes A)w_\psi \right),\\
    a\mapsto (v\otimes u\mapsto(u\otimes a)v),\quad a\in A,\ v\in W(\psi)_\lambda,\ u\in \g_{-\alpha}.
  \end{gather*}
  Since $\g_{-\alpha} = \C Y_\alpha$, Proposition~\ref{fin-dim-even} implies that $(\g_{-\alpha}\otimes A)w_\psi$ is finite dimensional. Thus, $I_{\alpha}$ is a linear subspace of $A$ of finite codimension. We claim that $I_{\alpha}$ is, in fact, an ideal of A. Indeed, since $\alpha\neq 0$, we can choose $h\in\h$ such that $\alpha(h)\neq 0$. Then, for all $g\in A$, $a\in I_{\alpha}$, $v\in W(\psi)_{\lambda}$, and $u\in \g_{-\alpha}$, we have
  \[
    0 = (h\otimes g)(u\otimes a)v
    =  [h\otimes g,u\otimes a]v + (u\otimes a)(h\otimes g)v
    =  -\alpha(h)(u\otimes ga)v + (u\otimes a)(h\otimes g)v.
  \]
  Since $(h\otimes g)v\in W(\psi)_{\lambda}$ and $a\in I_{\alpha}$, the last term above is zero. Since we also have $\alpha(h)\neq 0$, this implies that $(u\otimes ga)v=0$. As this holds for all $v\in W(\psi)_{\lambda}$ and $u\in \g_{-\alpha}$, we have $ga\in I_{\alpha}$. Hence $I_{\alpha}$ is an ideal of $A$.

  Let $I$ be the intersection of all the $I_{\alpha}$, $\alpha \in \Delta_{\bar 0}^+$.  Since $\g$ has a finite number of positive roots, this intersection is finite, and thus $I$ is also an ideal of $A$ of finite-codimension.  We have
  \[
    (\mathfrak{n}_{\bar 0}^- \otimes I)W(\psi)_{\lambda}=0 \quad \text{and} \quad (\mathfrak{n}^+ \otimes A)W(\psi)_{\lambda}=0.
  \]
  Then, since $\h\otimes I\subseteq [\n^{+}\otimes A,\n_{\bar 0}^{-}\otimes I]$, we have $(\h\otimes I)W(\psi)_{\lambda}=0$. In particular, $(\mathfrak{h} \otimes I)w_\psi=0$.

  Assume $\psi\notin \cL (\h\otimes A)$. Then there exists $a\in I$ such that $\psi(h\otimes a)\ne 0$ for some $h\in\h$, which implies that $w_\psi=0$, since
  \[
    0=(h\otimes a)w_\psi=\psi(h\otimes a)w_\psi.
  \]
  Therefore $W(\psi)=0$.
\end{proof}

\begin{defn}[The ideal $I_\psi$] \label{def:updated}
  For $\psi \in (\h \otimes A)^*$ with $\psi |_\h \in \Lambda^+$, let $I_\psi$ be the sum of all ideals $I\subseteq A$ such that $(\mathfrak{n}_{\bar 0}^- \otimes I)w_\psi=0$.
\end{defn}

\begin{rmk}\label{rem:ann-even-root}
  It follows from the proof of Proposition~\ref{trivial-weyl} that $I_\psi$ has finite codimension in $A$ and that $(\h\otimes I_\psi)w_\psi=0$. Furthermore, by Lemma~\ref{lem:assoc-alg-facts}, parts \eqref{lem-item:support-power} and \eqref{lem-item:finiteCod-and-finiteSupp}, since $I_\psi$ has finite codimension and $A$ is finitely generated, we have that $I_{\psi}^N$ has finite codimension, for all $N \in \N$.
\end{rmk}

For the remainder of the paper, we assume that
\begin{center}
  $\Sigma$ is a system of simple roots for $\g$ satisfying \eqref{eq:system-key-propery}.
\end{center}
Recall that, by Proposition~\ref{prop:nice-system-exists}, such a system always exists under our current assumption that $\g$ is either a basic classical Lie superalgebra or $\fsl(n,n)$, $n \ge 2$.

\begin{lemma} \label{fin.dim.odd}
  Suppose $\psi \in (\h \otimes A)^*$ with $\psi |_\h \in \Lambda^+$.  Then there exists $N_\psi\in\N$ such that
  \[
    \left( \n^-\otimes I_\psi^{N_\psi} \right) w_\psi=0.
  \]
\end{lemma}

\begin{proof}
  Recall the set $\{Y_\alpha \mid \alpha \in \Sigma\}$ of generators of $\n^-$.  We claim that
  \begin{equation} \label{eq:Ybeta-kill-claim}
    (Y_\alpha \otimes I_\psi) w_\psi=0, \quad \text{for all } \alpha \in \Sigma.
  \end{equation}
   By the definition of $I_\psi$, it suffices to consider the case $\alpha \in \Sigma_{\bar 1}$.  Fix such an $\alpha$.  By \eqref{eq:system-key-propery}, there exists $\alpha' \in \Delta_{\bar 1}$ such that $\beta := \alpha + \alpha' \in \Delta_{\bar 0}^+$.

  First suppose $\g$ is not $A(1,1)$ or $\fsl(2,2)$.  Then $\dim \g_\nu=1$ for any $\nu \in \Delta$ (see Remark~\ref{rmk:dim1-root-space}).  Thus, rescaling if necessary,
  \begin{equation}\label{eq:rmk-on-A(1,1)}
    [X_{\alpha'}, Y_\beta] = Y_\alpha.
  \end{equation}
  \details{
    Here we use the fact that $[\g_{\alpha'},\g_{\alpha''}] = \g_{\alpha' + \alpha''}$ if $\alpha', \alpha'', \alpha' + \alpha'' \in \Delta$.
  }
  Then,
  \[
    (Y_\alpha \otimes I_\psi)w_\psi = [X_{\alpha'}\otimes A,Y_\beta \otimes I_\psi]w_\psi
    \subseteq (X_{\alpha'}\otimes A) (Y_\beta \otimes I_\psi)w_\psi +(Y_\beta \otimes I_\psi)(X_{\alpha'}\otimes A)w_\psi = 0,
  \]
  where the last equality follows from the fact that $(Y_\beta \otimes I_\psi) w_\psi = 0$ by the definition of $I_\psi$ and $(X_{\alpha'} \otimes A) w_\psi = 0$ by the first relation in \eqref{def:local-Weyl}.  This proves \eqref{eq:Ybeta-kill-claim}.

  To prove \eqref{eq:Ybeta-kill-claim} for $\fsl(2,2)$ and $A(1,1)$, we consider $\g=\gl(2,2)$ and we let $\h$ be the subalgebra of diagonal matrices of $\g$. Denote by $\{\epsilon_i \mid i=1,\dotsc, 4\}$ the basis of $\h^*$ dual to $\{E_{i,i}\mid i=1,\ldots, 4\}$. In this case,
  \[
    \Delta_{\bar 0} = \{\pm(\epsilon_1 - \epsilon_2), \pm(\epsilon_3 - \epsilon_4)\},\quad \Delta_{\bar 1} = \{\pm(\epsilon_1-\epsilon_3),\pm(\epsilon_1-\epsilon_4),\pm(\epsilon_2-\epsilon_3), \pm(\epsilon_2-\epsilon_4)\},
  \]
  and $\g_{\epsilon_r-\epsilon_s} = \C E_{r,s}$, for $1 \leq r \neq s \leq 4$. In particular, if we fix $\alpha\in\Sigma_{\bar 1}$ and $\alpha'\in\Delta_{\bar 1}^+$ such that $\beta := \alpha + \alpha' \in \Delta$, then there exist $k,\ell,p,q\in\{1,2,3,4\}$ with $k\neq \ell$ and $p\neq q$, such that $\g_{\alpha'} = \C E_{k,\ell}$ and $\g_{-\beta} = \C E_{p,q}$. Since $\beta \in \Delta_{\bar 0}^+$ and $\alpha' \in \Delta_{\bar 1}^+$,  the definition of $I_\psi$ and the first relation in \eqref{def:local-Weyl} give us that $([E_{k,\ell}, E_{p,q}]\otimes I_\psi)w_\psi=0$. Regarding the $\fsl(2,2)$ case, we choose $Y_{\alpha}=[E_{k,\ell}, E_{p,q}]$. For the $A(1,1)$ case, we choose $Y_{\alpha}$ to be the image of $[E_{k,\ell}, E_{p,q}]$ in $A(1,1)$.  Then $(Y_\alpha\otimes I_\psi)w_\psi=0$.  Since the choice of $\alpha\in\Sigma_{\bar 1}$ was arbitrary, we conclude that $(Y_\alpha \otimes I_\psi)w_\psi=0$ for all roots $\alpha \in \Sigma_{\bar 1}$.

  Now, for $\beta = \sum_{\alpha \in \Sigma}^n m_\alpha \alpha \in \Delta^+$, we define the \emph{height} of $\beta$ to be $\het \beta := \sum_{\alpha \in \Sigma}^n m_\alpha$.  We prove, by induction on the height of $\beta$, that $(Y_{\beta} \otimes I_\psi^{\het \beta}) w_\psi = 0$ for all $\beta \in \Delta^+$.  Since $\g$ is finite dimensional, the heights of elements of $\Delta^+$ are bounded above, and thus the lemma will follow.

  The base case of height one is precisely \eqref{eq:Ybeta-kill-claim}.  Suppose $\beta \in \Delta^+$ with $\het \beta > 1$.  Then there exist $\beta', \beta'' \in \Delta^+$ with $\het \beta', \het \beta'' < \het \beta$ such that $Y_{\beta} \in \C [Y_{\beta'}, Y_{\beta''}]$.  Then
  \[
    (Y_{\beta} \otimes I_\psi^{\het \beta})w_\psi
    = [Y_{\beta'} \otimes I_\psi^{\het \beta'}, Y_{\beta''} \otimes I_\psi^{\het \beta''}] w_\psi
    = 0. \qedhere
  \]
\end{proof}

\begin{cor}\label{ind-step}
  Suppose $\psi \in (\h \otimes A)^*$ with $\psi |_\h \in \Lambda^+$, and let $N_\psi$ be as in Lemma~\ref{fin.dim.odd}.  Then
  \[
    \left( \g\otimes I_\psi^{N_\psi} \right)w_\psi=0.
  \]
\end{cor}

\begin{proof}
  It follows from the first relation in \eqref{def:local-Weyl} that $(\n^+\otimes I_\psi^{N_\psi})w_\psi=0$.  Since  $(\h\otimes I_\psi)w_\psi=0$ by Remark~\ref{rem:ann-even-root}, we have $ (\h\otimes I_\psi^{N_\psi})w_\psi=0$.   Finally Lemma~\ref{fin.dim.odd} implies that $(\n^-\otimes I_\psi^{N_\psi})w_\psi=0$.
\end{proof}

\begin{lemma} \label{lem:Wlambda-finite-weights}
  For all $\psi \in (\h \otimes A)^*$ with $\psi|_\h \in \Lambda^+$, the set of $\g$-weights (equivalently, $\g_{\bar 0}$-weights) of $W(\psi)$ is finite.
\end{lemma}

\begin{proof}
  Since the weights of $W(\lambda)$ are contained in $\lambda - Q^+$, a finite number of weights of $W(\lambda)$ are dominant integral.  Since $W(\lambda)$ is a direct sum of $\g_{\bar 0}$-modules, its weights are invariant under the (finite) Weyl group of $\g_{\bar 0}$.  The result follows.
\end{proof}

\begin{thm} \label{thm:fin-dim-loc-weyl}
  Assume that $A$ is finitely generated and the system of simple roots $\Sigma$ satisfies \eqref{eq:system-key-propery}.  Then the local Weyl module $W(\psi)$ is finite dimensional for all $\psi\in (\h\otimes A)^*$ such that $\psi |_{\h} \in \Lambda^+$.
\end{thm}

\begin{proof}
  By Definition~\ref{def:local-Weyl}, we have $W(\psi)=U(\n^-\otimes A)w_\psi$.  By Lemma~\ref{fin.dim.odd}, we have $( \n^-\otimes I_\psi^{N_\psi} )w_\psi=0$.  Thus $W(\psi) = U ( \n^-\otimes A/I_\psi^{N_\psi} ) w_\psi$.  By Lemma~\ref{lem:Wlambda-finite-weights}, there exists $N \in \N$ such that
  \[
    W(\psi) = U_n \left( \n^- \otimes A/I_\psi^{N_{\psi}} \right) w_\psi, \quad \text{for all } n\geq N,
  \]
  where $U(\mathfrak{a}) = \sum_{n=0}^\infty U_n(\mathfrak{a})$ is the usual filtration on the universal enveloping algebra of a Lie superalgebra $\mathfrak{a}$ induced from the natural grading on the tensor algebra.  Since the Lie superalgebra $\n^-\otimes A/I_\psi^{N_\psi}$ is finite dimensional (see Remark~\ref{rem:ann-even-root}), $W(\psi)$ is also finite dimensional.
\end{proof}

In the non-super setting, Theorem~\ref{thm:fin-dim-loc-weyl} was proved in \cite[Th.~1]{CP01} for $A = \C[t,t^{-1}]$, and in \cite[Th.~1]{FL04} for $A$ the algebra of functions on a complex affine variety.

\begin{prop} \label{prop:local-Weyl-universal}
  Let $\psi\in \cL (\h\otimes A)$ be such that $\psi|_{\h}=\lambda\in \Lambda^+$. Then the local Weyl module $W(\psi)$ is the unique (up to isomorphism) finite-dimensional object of $\I(\g\otimes A,\g_{\bar 0})$ that is generated by a highest map-weight vector of map-weight $\psi$ and admits a surjective homomorphism to any finite-dimensional object of $\I(\g\otimes A,\g_{\bar 0})$ also generated by a highest map-weight vector of map-weight $\psi$.
\end{prop}

\begin{proof}
  Let $V$ be a finite-dimensional object of $\I(\g\otimes A,\g_{\bar 0})$ that is generated by a highest map-weight vector $v$ of map-weight $\psi$.  It follows immediately from the definition of a highest map-weight $\g\otimes A$-module that the two first relations in \eqref{relations-local} are satisfied by $v$. Since the $\g_{\bar 0}$-module generated by $v$ must be finite dimensional, we have also that $Y_\alpha^{\lambda(H_\alpha)+1}v=0$, for all $\alpha \in \Sigma(\g_{\bar 0})$.  Therefore, there exists a surjective homomorphism $W(\psi) \to V$ sending $w_\psi$ to $v$.

  To show that $W(\psi)$ is the unique representation with the given property, suppose that $W$ is another module with this property.  Then $W$ is a quotient of $W(\psi)$ and vice-versa.  Since both modules are finite dimensional, it follows that $W(\psi) \cong W$.
\end{proof}

\begin{cor}
  Let $\psi\in \cL (\h\otimes A)$ such that $\psi|_{\h}=\lambda\in \Lambda^+$. Then the local Weyl module $W(\psi)$ is the maximal finite-dimensional quotient of the global Weyl module $W(\lambda)$ that is a highest map-weight module of highest map-weight $\psi$.
\end{cor}

By \cite[Th.~4.16]{Sav14}, any irreducible finite-dimensional $\g\otimes A$-module is a highest map-weight module, for some $\psi\in \cL  (\h\otimes A)$ with $\psi |_\h\in \Lambda^+$. Then, by Proposition~\ref{prop:local-Weyl-universal}, there exists a surjective homomorphism from the local Weyl module $W(\psi)$ to such an irreducible module.  In other words, all irreducible finite-dimensional $\g \otimes A$-modules are quotients of local Weyl modules.

We conclude by showing that the local Weyl modules possess a tensor product property analogous to the one satisfied in the non-super setting (see, \cite[Th.~2]{CP01} and \cite[Th.~2]{FL04}).

\begin{thm} \label{thm:tensor-property}
  Assume that $A$ is finitely generated and the system of simple roots $\Sigma$ satisfies \eqref{eq:system-key-propery}.  For $i=1,2$, let $\psi_i\in \cL (\h\otimes A)$ with $\lambda_i=\psi_i |_\h\in \Lambda^+$, and suppose that $I_{\psi_1}$ and $I_{\psi_2}$ have disjoint support. Then
  \[
    W(\psi_1+\psi_2)\cong W(\psi_1)\otimes W(\psi_2)
  \]
  as $\g\otimes A$-modules.
\end{thm}

\begin{proof}
  By Corollary~\ref{ind-step}, there exist $N_1, N_2 \in \N$ such that $( \g\otimes I_{\psi_i}^{N_i} ) w_{\psi_i}=0$ for $i=1,2$. Then the action of $\g\otimes A$ on $W(\psi_1)\otimes W(\psi_2)$ factors through the composition
  \begin{equation} \label{eq:tensor-factor}
    \g\otimes A \stackrel{d}{\hookrightarrow} (\g\otimes A)\oplus (\g\otimes A)\stackrel{\pi}{\twoheadrightarrow} (\g\otimes A/I_{\psi_1}^{N_1})\oplus (\g\otimes A/I_{\psi_2}^{N_2}),
  \end{equation}
  where $d$ is the diagonal embedding.  Since $\Supp(I_{\psi_1})\cap\Supp(I_{\psi_2}) = \varnothing$, Lemma~\ref{lem:assoc-alg-facts}\eqref{lem-item:support-power} implies that $\Supp(I_{\psi_1}^{N_1}) \cap \Supp(I_{\psi_2}^{N_2})=\varnothing$. Then, by Lemma~\ref{lem:assoc-alg-facts}\eqref{lem-item:product-intersection}, we have $A=I_{\psi_1}^{N_1}+I_{\psi_2}^{N_2}$ and $I_{\psi_1}^{N_1}\cap I_{\psi_2}^{N_2}=I_{\psi_1}^{N_1}  I_{\psi_2}^{N_2}$.  Thus, $A/I_{\psi_1}^{N_1} I_{\psi_2}^{N_2}\cong (A/I_{\psi_1}^{N_1})\oplus (A/I_{\psi_2}^{N_2})$. We therefore have the following commutative diagram:
  \[
    \xymatrix{\g\otimes A \ar@{->>}[d]  \ar@{^{(}->}[r]^-{d} & (\g\otimes A)\oplus (\g\otimes A) \ar@{->>}[d] \\
      \g\otimes A/I_{\psi_1}^{N_1}I_{\psi_2}^{N_2} \ar[r]^-{\cong} & (\g\otimes A/I_{\psi_1}^{N_1})\oplus (\g\otimes A/I_{\psi_2}^{N_2})}
  \]
  It follows that the composition~\eqref{eq:tensor-factor} is surjective.

  Since $W(\psi_1)\otimes W(\psi_2)$ is generated as a $(\g\otimes A/I_{\psi_1}^{N_1})\oplus (\g\otimes A/I_{\psi_2}^{N_2})$-module by the vector $w_{\psi_1} \otimes w_{\psi_2}$, it follows from the above that it is also generated by this vector as a $\g \otimes A$-module.  Moreover, $\h\otimes A$ acts on $w_{\psi_1}\otimes w_{\psi_2}$ via $\psi := \psi_1 + \psi_2$.  Thus $W(\psi_1)\otimes W(\psi_2)$ is a finite-dimensional highest map-weight module of highest map-weight $\psi$.  Therefore, by Proposition~\ref{prop:local-Weyl-universal}, it is a quotient of $W(\psi)$.

  To simplify notation, let $I_1 = I_{\psi_1}$, $I_2 = I_{\psi_2}$ and $N=N_\psi$.  Let $I = I_1 I_2 = I_1 \cap I_2$.  Then $I \subseteq I_\psi$.  Therefore,
  the action of $\bb \otimes A$ on $\C w_\psi$ descends to an action of $\bb\otimes A/I^N$ on $\C w_\psi$.  Consider the induced module
  \[
    M(\psi) := U(\g\otimes A/I^{N})\otimes_{U(\mathfrak{b}\otimes A/I^{N})} \C w_\psi.
  \]
  It follows from Corollary~\ref{ind-step} that $W(\psi)$ is a quotient of $M(\psi)$. On the other hand, it is clear that the one-dimensional $\mathfrak{b}\otimes A$-modules $\C w_\psi$ and $\C w_{\psi_1}\otimes \C w_{\psi_2}$ are isomorphic. Hence,
  \begin{align*}
    M(\psi) & =  U(\g\otimes A/I^N)\otimes_{U(\mathfrak{b}\otimes A/I^N)} \C w_\psi \\
    &\cong U \left( \g\otimes \left( A/I_1^{N}\oplus A/I_2^{N} \right) \right) \otimes_{U(\mathfrak{b} \otimes (A/I_1^{N} \oplus A/I_2^{N}))} \left( \C w_{\psi_1}\otimes \C w_{\psi_2} \right) \\
    &\cong \left( U \left( \g\otimes \left( A/I_1^{N} \right) \right) \otimes U \left( \g \otimes \left( A/I_2^{N} \right) \right) \right) \otimes_{U(\mathfrak{b}\otimes (A/I_1^{N}))\otimes U (\mathfrak{b}\otimes (A/I_2^{N}))} \left( \C w_{\psi_1}\otimes \C w_{\psi_2} \right) \\
    &\cong \left( U \left( \g\otimes \left( A/I_1^{N} \right) \right) \otimes_{U(\mathfrak{b} \otimes (A/I_1^{N}))} \C w_{\psi_1} \right) \otimes  \left( U\left( \g\otimes \left( A/I_2^{N} \right) \right) \otimes_{U(\mathfrak{b} \otimes (A/I_2^{N}))} \C w_{\psi_2} \right)  \\
    &=  M(\psi_1)\otimes M(\psi_2).
  \end{align*}
  So $W(\psi)$ is a quotient of $M(\psi_1)\otimes M(\psi_2)$.  Fix a surjection $\theta \colon M(\psi_1) \otimes M(\psi_2) \to W(\psi)$.

  We claim that the image of $M(\psi_1)_\mu \otimes M(\psi_2)_\nu$ under $\theta$ is zero except for a finite number of weights $\mu$ and $\nu$.  By Lemma~\ref{lem:Wlambda-finite-weights}, the set $D$ of weights occurring in $W(\psi)$ is finite.  Thus, the sets
  \[
    D_1 = (\lambda_1-Q^+) \cap (-\lambda_2+D+Q^+) \quad \text{and} \quad D_2=(\lambda_2-Q^+) \cap (-\lambda_1+D+Q^+)
  \]
  are also finite.  Since, for $i=1,2$, the weights of $M(\psi_i)$ are contained in $\lambda_i - Q^+$, the image of  $M(\psi_1)_\mu \otimes M(\psi_2)_{\nu}$ under $\theta$ is zero unless $\mu \in \lambda_1 - Q^+$, $\nu \in \lambda_2 - Q^+$ and $\mu + \nu \in D$.  Thus it is nonzero only if $\mu \in D_1$ and $\nu\in D_2$, and hence the claim is proved.

  For $i=1,2$, let $M(\psi_i)'$ be the submodule of $M(\psi_i)$ generated by the weight subspaces $M(\psi_i)_\mu$ with $\mu \notin D_i$, and let $\bar M(\psi_i)=M(\psi_i) / M(\psi_i)'$.  Then $W(\psi)$ is a quotient of $\bar M(\psi_1) \otimes \bar M(\psi_2)$.  Because $I_i$ has finite codimension and there are only a finite number of weights occurring in the quotient $\bar M(\psi_i)$, this module is a finite-dimensional highest map-weight module of highest map-weight $\psi_i$.
  \details{
    $\bar M(\psi_i)$ is a highest map-weight module of map-weight $\psi_i$ because it is a quotient of $M(\psi_i)$.  The proof that $\bar M(\psi_i)$ is finite dimensional is similar to the proof of Theorem~\ref{thm:fin-dim-loc-weyl}.
  }
  Then, by Proposition~\ref{prop:local-Weyl-universal}, it is a quotient of $W(\psi_i)$. Thus, $\bar M(\psi_1) \otimes \bar M(\psi_2)$ is a quotient of $W(\psi_1) \otimes W(\psi_2)$, which implies that $W(\psi)$ is a quotient of $W(\psi_1) \otimes W(\psi_2)$. Since the modules $W(\psi)$ and $W(\psi_1) \otimes W(\psi_2)$ are both finite dimensional, the fact that one is a quotient of the other implies the isomorphism in the statement of the theorem.
\end{proof}


\bibliographystyle{alphaurl}

\bibliography{bibliography}

\newcommand{\arxiv}[1]{\href{http://arxiv.org/abs/#1}{\tt
  arXiv:\nolinkurl{#1}}}\newcommand{\doi}[1]{DOI:
  \href{http://dx.doi.org/#1}{\tt \nolinkurl{#1}}}
\begin{thebibliography}{BHLW17}

\bibitem[BHLW17]{BHLW14}
A.~Beliakova, K.~Habiro, A.~D. Lauda, and B.~Webster.
\newblock Current algebras and categorified quantum groups.
\newblock {\em J. Lond. Math. Soc. (2)}, 95(1):248--276, 2017.
\newblock \arxiv{1412.1417}.
\newblock \href {https://doi.org/10.1112/jlms.12001}
  {\path{doi:10.1112/jlms.12001}}.

\bibitem[CFK10]{CFK10}
V.~Chari, G.~Fourier, and T.~Khandai.
\newblock A categorical approach to {W}eyl modules.
\newblock {\em Transform. Groups}, 15(3):517--549, 2010.
\newblock \arxiv{0906.2014}.
\newblock \href {https://doi.org/10.1007/s00031-010-9090-9}
  {\path{doi:10.1007/s00031-010-9090-9}}.

\bibitem[CFS08]{CFS08}
V.~Chari, G.~Fourier, and P.~Senesi.
\newblock Weyl modules for the twisted loop algebras.
\newblock {\em J. Algebra}, 319(12):5016--5038, 2008.
\newblock \arxiv{0704.3819}.
\newblock \href {https://doi.org/10.1016/j.jalgebra.2008.02.030}
  {\path{doi:10.1016/j.jalgebra.2008.02.030}}.

\bibitem[CMS16]{CMS14}
L.~Calixto, A.~Moura, and A.~Savage.
\newblock Equivariant map queer {L}ie superalgebras.
\newblock {\em Canad. J. Math.}, 68(2):258--279, 2016.
\newblock \arxiv{1412.5098}.
\newblock \href {https://doi.org/10.4153/CJM-2015-033-6}
  {\path{doi:10.4153/CJM-2015-033-6}}.

\bibitem[Cou16]{Cou14}
K.~Coulembier.
\newblock Bott-{B}orel-{W}eil theory and {B}ernstein-{G}el'fand-{G}el'fand
  reciprocity for {L}ie superalgebras.
\newblock {\em Transform. Groups}, 21(3):681--723, 2016.
\newblock \arxiv{1404.1416}.
\newblock \href {https://doi.org/10.1007/s00031-016-9377-6}
  {\path{doi:10.1007/s00031-016-9377-6}}.

\bibitem[CP01]{CP01}
V.~Chari and A.~Pressley.
\newblock Weyl modules for classical and quantum affine algebras.
\newblock {\em Represent. Theory}, 5:191--223 (electronic), 2001.
\newblock \arxiv{math/0004174}.
\newblock \href {https://doi.org/10.1090/S1088-4165-01-00115-7}
  {\path{doi:10.1090/S1088-4165-01-00115-7}}.

\bibitem[CW12]{CW12}
S.-J. Cheng and W.~Wang.
\newblock {\em Dualities and representations of {L}ie superalgebras}, volume
  144 of {\em Graduate Studies in Mathematics}.
\newblock American Mathematical Society, Providence, RI, 2012.

\bibitem[ER13]{Rao13}
S.~Eswara~Rao.
\newblock Finite dimensional modules for multiloop superalgebras of types
  {$A(m,n)$} and {$C(m)$}.
\newblock {\em Proc. Amer. Math. Soc.}, 141(10):3411--3419, 2013.
\newblock \arxiv{1109.3297}.
\newblock \href {https://doi.org/10.1090/S0002-9939-2013-11619-9}
  {\path{doi:10.1090/S0002-9939-2013-11619-9}}.

\bibitem[ERZ04]{RZ04}
S.~Eswara~Rao and K.~Zhao.
\newblock On integrable representations for toroidal {L}ie superalgebras.
\newblock In {\em Kac-{M}oody {L}ie algebras and related topics}, volume 343 of
  {\em Contemp. Math.}, pages 243--261. Amer. Math. Soc., Providence, RI, 2004.
\newblock \href {https://doi.org/10.1090/conm/343/06192}
  {\path{doi:10.1090/conm/343/06192}}.

\bibitem[FKKS12]{FKKS12}
G.~Fourier, T.~Khandai, D.~Kus, and A.~Savage.
\newblock Local {W}eyl modules for equivariant map algebras with free abelian
  group actions.
\newblock {\em J. Algebra}, 350:386--404, 2012.
\newblock \arxiv{1103.5766}.
\newblock \href {https://doi.org/10.1016/j.jalgebra.2011.10.018}
  {\path{doi:10.1016/j.jalgebra.2011.10.018}}.

\bibitem[FL04]{FL04}
B.~Feigin and S.~Loktev.
\newblock Multi-dimensional {W}eyl modules and symmetric functions.
\newblock {\em Comm. Math. Phys.}, 251(3):427--445, 2004.
\newblock \arxiv{math/0212001}.
\newblock \href {https://doi.org/10.1007/s00220-004-1166-8}
  {\path{doi:10.1007/s00220-004-1166-8}}.

\bibitem[FMS13]{FMS13}
G.~Fourier, N.~Manning, and P.~Senesi.
\newblock Global {W}eyl modules for the twisted loop algebra.
\newblock {\em Abh. Math. Semin. Univ. Hambg.}, 83(1):53--82, 2013.
\newblock \arxiv{1110.2752}.
\newblock \href {https://doi.org/10.1007/s12188-013-0074-2}
  {\path{doi:10.1007/s12188-013-0074-2}}.

\bibitem[FMS15]{FMS14}
G.~Fourier, N.~Manning, and A.~Savage.
\newblock Global {W}eyl modules for equivariant map algebras.
\newblock {\em Int. Math. Res. Not. IMRN}, (7):1794--1847, 2015.
\newblock \arxiv{1303.4437}.
\newblock \href {https://doi.org/10.1093/imrn/rnt231}
  {\path{doi:10.1093/imrn/rnt231}}.

\bibitem[FSS00]{FSS00}
L.~Frappat, A.~Sciarrino, and P.~Sorba.
\newblock {\em Dictionary on {L}ie algebras and superalgebras}.
\newblock Academic Press, Inc., San Diego, CA, 2000.

\bibitem[Kac77]{Kac77}
V.~G. Kac.
\newblock {L}ie superalgebras.
\newblock {\em Advances in Math.}, 26(1):8--96, 1977.

\bibitem[Kac78]{Kac78}
V.~Kac.
\newblock Representations of classical {L}ie superalgebras.
\newblock In {\em Differential geometrical methods in mathematical physics,
  {II}, ({P}roc. {C}onf., {U}niv. {B}onn, {B}onn, 1977)}, volume 676 of {\em
  Lecture Notes in Math.}, pages 597--626. Springer, Berlin, 1978.

\bibitem[KKT16]{KKT11}
S.-J. Kang, M.~Kashiwara, and S.~Tsuchioka.
\newblock Quiver {H}ecke superalgebras.
\newblock {\em J. Reine Angew. Math.}, 711:1--54, 2016.
\newblock \arxiv{1107.1039}.
\newblock \href {https://doi.org/10.1515/crelle-2013-0089}
  {\path{doi:10.1515/crelle-2013-0089}}.

\bibitem[Mus12]{Mus12}
I.~M. Musson.
\newblock {\em {L}ie superalgebras and enveloping algebras}, volume 131 of {\em
  Graduate Studies in Mathematics}.
\newblock American Mathematical Society, Providence, RI, 2012.

\bibitem[NS13]{NS13}
E.~Neher and A.~Savage.
\newblock A survey of equivariant map algebras with open problems.
\newblock In {\em Recent developments in algebraic and combinatorial aspects of
  representation theory}, volume 602 of {\em Contemp. Math.}, pages 165--182.
  Amer. Math. Soc., Providence, RI, 2013.
\newblock \arxiv{1211.1024}.
\newblock \href {https://doi.org/10.1090/conm/602/12024}
  {\path{doi:10.1090/conm/602/12024}}.

\bibitem[Sav14]{Sav14}
A.~Savage.
\newblock Equivariant map superalgebras.
\newblock {\em Math. Z.}, 277(1-2):373--399, 2014.
\newblock \arxiv{1202.4127}.
\newblock \href {https://doi.org/10.1007/s00209-013-1261-7}
  {\path{doi:10.1007/s00209-013-1261-7}}.

\bibitem[SVV17]{SVV14}
P.~Shan, M.~Varagnolo, and E.~Vasserot.
\newblock On the center of quiver {H}ecke algebras.
\newblock {\em Duke Math. J.}, 166(6):1005--1101, 2017.
\newblock \arxiv{1411.4392}.
\newblock \href {https://doi.org/10.1215/00127094-3792705}
  {\path{doi:10.1215/00127094-3792705}}.

\bibitem[Zha14]{Zha14}
H.~Zhang.
\newblock Representations of quantum affine superalgebras.
\newblock {\em Math. Z.}, 278(3-4):663--703, 2014.
\newblock \href {https://doi.org/10.1007/s00209-014-1330-6}
  {\path{doi:10.1007/s00209-014-1330-6}}.

\end{thebibliography}

\end{document}